\documentclass[11pt]{amsart}
\usepackage{amsfonts}
\usepackage{graphicx}
\usepackage{amscd}
\usepackage{amsmath}
\usepackage{amssymb}
\usepackage{amsfonts}
\newtheorem{theorem}{Theorem}
\theoremstyle{plain}
\newtheorem{corollary}{Corollary}

\newtheorem{lemma}{Lemma}

\theoremstyle{definition}

\numberwithin{equation}{section}

\begin{document}
\title[Chebyshev Estimates for Generalized Primes]{Chebyshev Estimates for Beurling Generalized Prime Numbers. I}
\author[J. Vindas]{Jasson Vindas}
\address{Department of Mathematics, Ghent University, Krijgslaan 281 Gebouw S22, B 9000 Gent, Belgium}
\email{jvindas@cage.Ugent.be}
\thanks{The author gratefully acknowledges support by a Postdoctoral Fellowship of the Research
Foundation--Flanders (FWO, Belgium)}


\begin{abstract}
We provide new sufficient conditions for Chebyshev estimates for Beurling generalized primes. It is shown that if the counting function $N$ of a generalized number system satisfies the $L^{1}$-condition 
$$
\int_{1}^{\infty}\left|\frac{N(x)-ax}{x}\right|\frac{\mathrm{d}x}{x}<\infty
$$
and 
$N(x)=ax+o\left(x/\log x\right),$
for some $a>0$, then
$$
0<\liminf_{x\to\infty}\frac{\psi(x)}{x}\ \ \ \mbox{and}\ \ \ \limsup_{x\to\infty}\frac{\psi(x)}{x}<\infty
$$
hold. We give an analytic proof of this result. It is based on the Wiener division theorem. Our result extends those of Diamond (Proc. Amer. Math. Soc. 39 (1973), 503--508) and Zhang (Proc. Amer. Math. Soc. 101 (1987), 205--212).
\end{abstract}
\maketitle
\section{Introduction}
This note reports on new conditions that ensure the validity of Chebyshev estimates for Beurling's generalized primes. We will considerably improve earlier results by Diamond \cite{diamond3} and Zhang \cite{zhang}. In particular, we shall answer an open question posed by Diamond in \cite[p. 10]{diamond4}.

Let $P=\left\{p_k\right\}_{k=1}^{\infty}$ be a set of Beurling generalized primes, that is, a non-decreasing sequence of real numbers tending to infinity, where it is assumed $p_1>1$. The sequence $\left\{n_{k}\right\}_{k=1}^{\infty}$ denotes its associated set of generalized integers \cite{bateman-diamond,beurling}. Set further,
\begin{equation*}
N(x)=N_{P}(x)=\sum_{n_{k}<x}1 \ \ \mbox{ and }\ \ 
\psi(x)=\psi_{P}(x)=\sum_{n_{k}<x}\Lambda(n_{k})\ ,
\end{equation*}
where $\Lambda=\Lambda_{P}$ is the von Mangoldt function of the generalized number system \cite{bateman-diamond}. 
Beurling \cite{beurling} investigated the truth of the prime number theorem (PNT) in this context, i.e.,
$$
\psi(x)\sim x\ , \ \ \ x\to\infty\ .
$$ 
He proved that 
\begin{equation}
\label{ibpneq4}
N(x)=ax+O\left(\frac{x}{\log^{\gamma}x}\right)\ , \ \ \ x\to\infty\ \ \  (a>0)\ ,
\end{equation}
where $\gamma>3/2$, suffices for the PNT to hold. His condition is sharp: When 
$\gamma=3/2$, then the PNT need not hold, as exhibited by counterexamples in \cite{beurling,diamond2}. See \cite{kahane1,vindasGPNT} for the most recent extensions of Beurling's PNT. 
Since the PNT breaks down for $\gamma\leq3/2$, a natural question arises: Under which conditions over $N$ do Chebyshev estimates hold true?
A partial answer to this question was provided by Diamond \cite{diamond3}, he showed that (\ref{ibpneq4}) with $\gamma>1$ is enough to obtain Chebyshev estimates, namely,
\begin{equation}
\label{pnbpneq5}
0<\liminf_{x\to\infty} \frac{\psi(x)}{x} \ \ \ \mbox{and} \ \ \ \limsup_{x\to\infty} \frac{\psi(x)}{x}<\infty\ .
\end{equation}
On the other hand, (\ref{pnbpneq5}) is not generally true when $\gamma<1$, as follows from an example of Hall \cite{hall}. 

Diamond conjectured \cite{diamond4} that  
\begin{equation}
\label{icepneq6}
\int_{1}^{\infty}\left|\frac{N(x)-ax}{x}\right|\frac{\mathrm{d}x}{x}<\infty\ ,  \ \ \ \mbox{with } a>0\ ,
\end{equation}
would be enough for (\ref{pnbpneq5}) to hold. His conjecture turned out to be false. Kahane established the negative answer to Diamond's conjecture in \cite{kahane2}. 

Remarkably, as shown in this article, if one adds a side condition to (\ref{icepneq6}), one can indeed achieve Chebyshev estimates. Our main goal is to prove the following theorem.

\begin{theorem}
\label{icepnth2} Diamond's $L^{1}$-condition and the asymptotic behavior
\begin{equation}
\label{icepneq7}
N(x)=ax+o\left(\frac{x}{\log x}\right)\ , \ \ \ x\to\infty\ ,
\end{equation}
suffice for Chebyshev estimates $(\ref{pnbpneq5})$.
\end{theorem}

Clearly, Theorem \ref{icepnth2} extends the result of Diamond quoted above. It should be noticed that Zhang \cite{zhang} also gave an extension of Diamond's theorem. Our result includes it as a particular instance:  
\begin{corollary}[Zhang \cite{zhang}]
\label{icec1} The Chebyshev estimates $(\ref{pnbpneq5})$ hold if
\begin{equation}
\label{icepneq5} \int_{1}^{\infty}\left( \sup_{x\leq t}\left|\frac{N(t)-at}{t}\right|\right)\frac{\mathrm{d}x}{x}<\infty \ \ \ (a>0)\ .
\end{equation}
\end{corollary}
\begin{proof} Naturally, (\ref{icepneq5}) implies (\ref{icepneq6}). If $\omega$ is a non-increasing and non-negative function such that $\int_{1}^{\infty}\omega(x)x^{-1}\mathrm{d}x<\infty$, one must have $\omega(x)=o(1/\log x)$; thus, Zhang's condition (\ref{icepneq5}) always yields (\ref{icepneq7}).
\end{proof}

We shall give a proof of Theorem \ref{icepnth2} in Section \ref{cebpnproof}. We point out that the methods of Diamond and Zhang from \cite{diamond3,zhang} are elementary. Furthermore,  Diamond has asked in \cite[p. 10]{diamond4} whether it is possible to find an analytic approach to Chebyshev inequalities. Our proof of Theorem \ref{icepnth2} is non-elementary and it therefore gives an answer to Diamond's question; it uses the zeta function of the generalized number system and the Wiener division theorem \cite[Chap. 2]{korevaar}. In addition, we make use of the operational calculus for the Laplace transform of distributions \cite{vladimirov}. It would also be interesting to find an elementary proof. Finally, it should be mentioned that Zhang has provided another related sufficient condition for Chebyshev estimates in \cite{zhang1993}; the author announces that is also possible to obtain substantial improvements of that result. However, we will not pursue those questions here and they will be treated elsewhere. 
\subsection{Notation}

The Schwartz spaces $\mathcal{D}(\mathbb{R})$, $\mathcal{S}(\mathbb{R})$, $\mathcal{D}'(\mathbb{R})$ and $\mathcal{S}'(\mathbb{R})$ are well known; we refer to \cite{vladimirov} for their properties.  
If $f\in\mathcal{S}'(\mathbb{R})$ has support in $[0,\infty)$, its Laplace transform is well defined as
$$
\mathcal{L}\left\{f;s\right\}=\left\langle f(u),e^{-su}\right\rangle\ , \ \ \  \Re e\:s>0\ ,
$$
and the Fourier transform $\hat{f}$ is the distributional boundary value of $\mathcal{L}\left\{f;s\right\}$ on $\Re e\:s=0$.

We use the notation $H$ for the \emph{Heaviside function}, it is simply the characteristic function of $(0,\infty)$.

\section{Proof of Theorem \ref{icepnth2}}\label{cebpnproof}
We assume (\ref{icepneq6}) and (\ref{icepneq7}). Our starting point is the identity
\begin{equation}
\label{cepnpeq1}
\mathcal{L}\left\{\psi(e^{u});s\right\}=-\frac{\zeta'(s)}{s\zeta(s)}= \frac{1}{s}\cdot \frac{-(s-1)G'(s)}{(s-1)\zeta(s)}-\frac{G(s)}{s(s-1)\zeta(s)}-\frac{1}{s}+\frac{1}{s-1}\ ,
\end{equation}
where
$$
G(s):=\zeta(s)-\frac{a}{s-1}\ .
$$
We set $E_{1}(u):=e^{-u}N(e^{u})-aH(u)$. Our assumptions (\ref{icepneq6}) and (\ref{icepneq7}) translate into $E_{1}\in L^{1}(\mathbb{R})$ and $uE_{1}(u)=o(1)$, $u\to\infty$.
\begin{lemma}\label{cepnl1} $G(s)$ extends to a continuous function on $\Re e\:s=1$. Consequently, $(s-1)\zeta(s)$ is continuous on $\Re e\:s=1$ and there exists $c>0$ such that $it\zeta(1+it)\neq 0$ for all $t\in(-3c,3c)$.
\end{lemma}
\begin{proof} Clearly, 
\begin{equation}
\label{cepnpeq0}
G(s)= s\mathcal{L}\left\{E_{1};s-1\right\}+a\ .
\end{equation}
Taking distributional boundary values, we obtain $G(1+it)=(1+it)\hat{E}_{1}(t)+a$. Since $E_1\in L^{1}(\mathbb{R})$, $\hat{E}_{1}$ is continuous and the assertions follow at once. 
\end{proof}
Set $T(u)=e^{-u}\psi(e^{u})$, we must show that
\begin{equation}
\label{cepnpeq2}
0<\liminf_{u\to\infty}T(u) \ \ \ \mbox{and} \ \ \ \limsup_{u\to\infty}T(u)<\infty\ . 
\end{equation}
The next step is to use the boundary behavior of (\ref{cepnpeq1}) near $s=1$ to derive (\ref{cepnpeq2}). We first study convolution averages of $T$.

\begin{lemma}
\label{cepnl2} For any fixed $\phi\in \mathcal{D}(-c,c)$,
\begin{equation}
\label{cepnpeq3}
\int_{-\infty}^{\infty}T(u)\hat{\phi}(u-h)\mathrm{d}u=\int_{-\infty}^{\infty}\hat{\phi}(u)\mathrm{d}u+o(1)\ , \ \ \ h\to\infty\ .
\end{equation}
\end{lemma} 
\begin{proof}
Fix $\phi\in\mathcal{D}(-c,c)$. We use (\ref{cepnpeq0}) to decompose (\ref{cepnpeq1}) further, 
\begin{equation}
\label{cepnpeq5}
-\frac{\zeta'(s)}{s\zeta(s)}= \frac{(s-1)\mathcal{L}\left\{uE_{1}(u);s-1\right\}}{(s-1)\zeta(s)}-\frac{\mathcal{L}\left\{E_{1};s-1\right\}}{s\zeta(s)}-\frac{G(s)}{s(s-1)\zeta(s)}-\frac{1}{s}+\frac{1}{s-1}\ .
\end{equation}
Set now
$$
g_{1}(t):=\lim _{\sigma\to1^{+}}\frac{(\sigma-1+it)\mathcal{L}\left\{uE_{1}(u);\sigma-1+it\right\}}{(\sigma-1+it)\zeta(\sigma+it)} \ \ \ \mbox{in }\mathcal{S}'(\mathbb{R})\ , 
$$ 
and 
$$
g_{2}(t):=-\lim _{\sigma\to1^{+}}\left(\frac{(\sigma-1+it)\mathcal{L}\left\{E_{1};\sigma-1+it\right\}+G(\sigma+it)}{(\sigma+it)(\sigma-1+it)\zeta(\sigma+it)}+\frac{1}{\sigma+it}\right)
$$
in $\mathcal{S}'(\mathbb{R})$. Taking boundary values in (\ref{cepnpeq5}), 
we obtain
$
\hat{T}(t)=g_{1}(t)+ g_{2}(t)+\hat{H}(t),
$
an equality that must be interpreted in the sense of distributions. Recall that $H$ is the Heaviside function. By Lemma \ref{cepnl1}, $g_{2}$ is continuous on $(-3c,3c)$. Next, applying the Riemann-Lebesgue lemma to the continuous function $\phi (t)g_{2}(t)$, we conclude that
\begin{align*}
\int_{-\infty}^{\infty}T(u)\hat{\phi}(u-h)\mathrm{d}u&=\left\langle \hat{T}(t),e^{iht}\phi(t)\right\rangle
\\
&
=\int_{-\infty}^{\infty}\hat{\phi}(u)\mathrm{d}u+ \left\langle g_{1}(t),e^{iht}\phi(t)\right\rangle+o(1)\ .
\end{align*}
Thus, it is enough to show that
$$
\lim_{h\to\infty} \left\langle g_{1}(t),e^{iht}\phi(t)\right\rangle=0\ .
$$
Let $M\in\mathcal{S}'(\mathbb{R})$ be the distribution supported in the interval $[0,\infty)$ that satisfies $\mathcal{L}\left\{M;s-1\right\}=((s-1)\zeta(s))^{-1}$. Notice also that $(s-1)\mathcal{L}\left\{E_{2};s-1\right\}=\mathcal{L}\left\{E_{2}';s-1\right\}$, where $E_{2}(u)=uE_1(u)=o(1)$, so we have that $g_{1}=\widehat{( E_{2}'\ast M)}$, where $\ast$ denotes convolution \cite{vladimirov}. Consider an even function $\eta\in \mathcal{D}(-3c,3c)$ such that $\eta(t)=1$ for all $t\in(-2c,2c)$. 
Clearly $\eta(t)it\zeta(1+it)\neq0$ for all $t\in(-2c,2c)$; moreover, it is the Fourier transform of an $L^{1}$-function. 
Finally, we apply the Wiener division theorem \cite[p. 88]{korevaar} to $\eta(t)it\zeta(1+it)$ and $\phi(t)$ and conclude the existence of $f\in L^{1}(\mathbb{R})$ such that

$$
\hat{f}(t)=\frac{\phi(t)}{\eta(t) it \zeta(1+it)}\ .
$$
Therefore, as $h\to\infty$,
$$\left\langle g_{1}(t),e^{iht}\phi(t)\right\rangle=\left\langle (E_{2}'\ast M)(u),\hat{\phi}(u-h)\right\rangle=(E_{2}\ast (\hat{\eta})'\ast f)(h)=o(1)\ , $$
because $E_{2}(u)=o(1)$ and $(\hat{\eta})'\ast f\in L^{1}(\mathbb{R})$. Thus, (\ref{cepnpeq3}) has been established. 
\end{proof}

The estimates (\ref{cepnpeq2}) follow now easily from (\ref{cepnpeq3}) in Lemma \ref{cepnl2}. Choose $\phi\in\mathcal{D}(-c,c)$ in (\ref{cepnpeq3}) such that $\hat{\phi}$ is non-negative.
Using the fact that $\psi(e^{u})$ is non-decreasing, we have that $e^{-u}T(h)\leq T(u+h)$ whenever $u$ and $h$ are positive, setting $C_1=\int_{0}^{\infty}e^{-u}\hat{\phi}(u)\mathrm{d}u>0,$
\begin{equation*}T(h)= 
C_1^{-1}\int_{0}^{\infty}e^{-u}T(h)\hat{\phi}(u)\mathrm{d}u\leq C_1^{-1}\int_{0}^{\infty}T(u+h)\hat{\phi}(u)\mathrm{d}u=O(1)\leq C_{2}\ ,
\end{equation*}
for some constant $C_{2}>0$. Fix now $A>0$; observe that if $u\leq A$, then $T(h)\geq e^{u-A}T(h-A+u)$, and hence
\begin{align*}
\liminf_{h\to\infty} T(h)&\geq \frac{e^{-A}}{\int_{-A}^{A}e^{-u}\hat{\phi}(u)\mathrm{d}u}\liminf_{h\to\infty}\int_{-A}^{A}T(h-A+u)\hat{\phi}(u)\mathrm{d}u\\
&
=\frac{e^{-A}}{\int_{-A}^{A}e^{-u}\hat{\phi}(u)\mathrm{d}u}\liminf_{h\to\infty}\left(\int_{-\infty}^{\infty}-\int_{\left|u\right|\geq A}\right)T(h-A+u)\hat{\phi}(u)\mathrm{d}u
\\
&
\geq 
\frac{e^{-A}}{\int_{-A}^{A}e^{-u}\hat{\phi}(u)\mathrm{d}u}\left(\int_{-\infty}^{\infty}\hat{\phi}(u)\mathrm{d}u-C_2\int_{\left|u\right|\geq A}\hat{\phi}(u)\mathrm{d}u\right)\ .
\end{align*}
It remains to choose $A$ so large that $\int_{-\infty}^{\infty}\hat{\phi}(u)\mathrm{d}u-C_2\int_{\left|u\right|\geq A}\hat{\phi}(u)\mathrm{d}u>0$. The proof is complete.

\end{document}